\documentclass{article}
\usepackage[left=3cm,right=3cm,top=3cm,bottom=3cm]{geometry} 
\usepackage{amsmath,amssymb} 
\usepackage{graphicx}
\usepackage{amsthm}
\usepackage{float}
\usepackage{caption}
\usepackage{subcaption}
\usepackage{tikz}

\newcommand{\sF}{\mathcal F}

\newcommand{\sI}{\mathcal I}

\newcommand{\sL}{\mathcal L}

\newcommand{\sN}{\mathcal N}
\newcommand{\sP}{\mathcal P}

\newcommand{\sS}{\mathcal S}
\newcommand{\sT}{\mathcal T}

\newcommand{\R}{\mathbb R}
\newcommand{\E}{\mathbb E}
\newcommand{\F}{\mathbb F}
\newcommand{\Q}{\mathbb Q}

\newcommand{\T}{\mathbb T}
\newcommand{\Prob}{\mathbb P}

\newcommand{\argmax}{\mbox{argmax}}

\newtheorem{thm}{Theorem}[section]
\newtheorem{prop}[thm]{Proposition}
\newtheorem{eg}[thm]{Example}
\newtheorem{lem}[thm]{Lemma}
\newtheorem{cor}[thm]{Corollary}
\newtheorem{rem}[thm]{Remark}
\newtheorem{sass}{Standing Assumption}

\begin{document}

\title{The shape of the value function under Poisson optimal stopping}
\author{David Hobson\thanks{Department of Statistics, University of Warwick, Coventry CV4 7AL, UK. d.hobson@warwick.ac.uk} }

\date{\today}

\maketitle

\begin{abstract}
In a classical problem for the stopping of a diffusion process $(X_t)_{t \geq 0}$, where the goal is to maximise the expected discounted value of a function of the stopped process $\E^x[e^{-\beta \tau}g(X_\tau)]$, maximisation takes place over all stopping times $\tau$. In a Poisson optimal stopping problem, stopping is restricted to event times of an independent Poisson process. In this article we consider whether the resulting value function $V_\theta(x) = \sup_{\tau \in \sT(\T^\theta)}\E^x[e^{-\beta \tau}g(X_\tau)]$ (where the supremum is taken over stopping times taking values in the event times of an inhomogeneous Poisson process with rate $\theta = (\theta(X_t))_{t \geq 0}$) inherits monotonicity and convexity properties from $g$. It turns out that monotonicity (respectively convexity) of $V_\theta$ in $x$ depends on the monotonicity (respectively convexity) of the quantity $\frac{\theta(x) g(x)}{\theta(x) + \beta}$ rather than $g$. Our main technique is stochastic coupling.



\end{abstract}

\section{Introduction} \label{sec:intro}
In a classical optimal stopping problem the objective is to maximise the expected discounted payoff, where the payoff is a function of some underlying process, typically a time-homogeneous diffusion, and the maximisation takes places over {\em all} stopping times. In a Poisson optimal stopping problem (Dupuis and Wang~\cite{DupuisWang:02}, Lempa~\cite{Lempa:12}, Lange et al~\cite{LangeRalphStore:19} --- the terminology was introduced by~\cite{LangeRalphStore:19}) the set of potential stopping times is restricted to be the set of event times of an independent Poisson process .  The idea behind introducing the Poisson optimal stopping problem is that in many applications (for example, the optimal time to sell a financial asset) there are restrictions on when stopping can occur (for example, liquidity restrictions may mean that buyers are not always available). If the underlying process to be stopped is Markovian, then it is very convenient (and also often realistic) to model the set of candidate opportunities to stop as the event times of a (not-necessarily homogeneous) Poisson process, as this will preserve the Markov property.
In this article we want to consider the properties of the solution to the Poisson optimal stopping problem, where we allow the rate of the Poisson process to depend on the underlying diffusion. Rather than studying a specific problem, we study a general class of problems, and look for general features of the value function.

Let $X$ be a diffusion process, $g$ a non-negative payoff function and $\beta$ an impatience factor. The classical optimal stopping problem is to find
\begin{equation}
w(x) = \sup_{\tau \in \sT([0,\infty))} \E^x [ e^{-\beta \tau} g(X_\tau)],
\label{eq:wdef}
\end{equation}
where $\sT( \T )$ is the set of all stopping times taking values in $\T$, and in this case $\T = [0,\infty)$. The Poisson optimal stopping problem, introduced by Dupuis and Wang~\cite{DupuisWang:02} in the case where $X$ is exponential Brownian motion and extended to general diffusion processes by Lempa~\cite{Lempa:12}, is to find
\begin{equation}
V_\lambda(x) = \sup_{\tau \in \sT(\T^\lambda)} \E^x[e^{-\beta \tau} g(X_\tau)]
\label{eq:hdef}
\end{equation}
where $\T^\lambda$ is the set of event times of a Poisson process with rate $\lambda$.

The Poisson optimal stopping problem has been extended in many ways and to many settings, for example to allow for regime switching (Liang and Wei~\cite{LiangWei:16}), non-exponential inter-arrival times (Menaldi and Robin~\cite{MenaldiRobin:16}) and running costs and multi-dimensions (Lange et al~\cite{LangeRalphStore:19}). A related work in which actions are constrained to occur only at event times of a Poisson process is Rogers and Zane~\cite{RogersZane:98} who model portfolio optimisation.

Hobson and Zeng~\cite{HobsonZeng:19} consider an extension of \eqref{eq:hdef} in which the agent can choose the rate of the Poisson process (dynamically) subject to a cost which depends on the chosen rate. Motivated by this example, in this paper we consider the extension of \eqref{eq:hdef} to a state-dependent, inhomogeneous Poisson process and the problem of finding
\begin{equation}
V_\theta(x) = \sup_{\tau \in \sT(\T^\theta)} \E^x[e^{-\beta \tau} g(X_\tau)]
\label{eq:Vdef}
\end{equation}
where $\T^\theta$ is the set of event times of a time-inhomogeneous Poisson process with rate $\theta(X_t)$ at time $t$. (We will use the symbol $\lambda$ in the case of a constant-rate Poisson process, and $\theta$ in the case of a state-dependent Poisson process, but essentially the only purpose of a different notation is to allow us to highlight the results in the constant rate case.) 

One approach to solving \eqref{eq:Vdef} (and also \eqref{eq:hdef}) is to use the Bellman-type representation
\begin{equation}
V_\theta(x) = \E^x[e^{-\beta T^\theta_1} \max \{ g(X_{T^\theta_1}), V_\theta(X_{T^\theta_1} \}]
\label{eq:Vdef2}
\end{equation}
where $T^\theta_1$ is the first event time of the Poisson process with rate $\theta = \{ \theta(X_t) \}_{t \geq 0}$. This representation is based on the fact that at the first event time of the Poisson process the agent chooses between stopping and continuing. Solving \eqref{eq:Vdef2}, even numerically, may be challenging as the unknown $V_\theta$ appears on both sides. One strategy, as described in Lange et al~\cite{LangeRalphStore:19} is as follows. Let $V_\theta^{(n)}$ denote the value function under the restriction that stopping is constrained to lie in the first $n$ events of the Poisson process. If we set $V^{(0)}_\theta=0$ then the family $(V^{(n)}_\theta)_{n \geq 1}$ solves
\begin{equation}
V^{(n)}_\theta(x) = \E^x[e^{-\beta T^\theta_1} \max \{ g(X_{T^\theta_1}), V^{(n-1)}_\theta(X_{T^\theta_1}) \}].
\label{eq:Vdefn}
\end{equation}
Since $V^{(1)}_\theta \geq 0 = V^{(0)}_\theta$ it is easy to see that $V^{(n)}_\theta$ is increasing in $n$ (this is also clear from the definition) and therefore $V_\theta^{(\infty)}$ defined by $V_\theta^{(\infty)}(x) = \lim_{n \uparrow \infty} V_\theta^{(n)}(x)$ exists. Moreover, since we expect that $V_\theta^{(\infty)} = V_\theta$ we have found our solution.

In this article we are concerned with the monotonicity and convexity in $x$ of $V_\theta(x)$. A secondary goal is to understand the relationship between $V_\theta^{(\infty)}$ and $V_\theta$. We give a simple sufficient condition for equality, but also an example to show that they are not always equal.

Temporarily, instead of an optimal stopping problem, consider a fixed-horizon problem: $U(x)= \E^x[ e^{- \beta \kappa}g(X_\kappa)]$ where $\kappa$ is a constant time. Suppose $g$ is increasing: a simple Doeblin coupling argument (see Lindvall~\cite[p24]{Lindvall:92}, Bergmann et al~\cite{BergmanGrundyWiener:96}, Henderson et al~\cite{HendersonSunWhalley:14}) gives that $U$ is also increasing. Further, if $X$ is exponential Brownian motion and $g$ is convex then $w$ is convex (Cox and Ross~\cite{CoxRoss:76}). Subject to the condition that $X$ is a martingale, this convexity result has been extended to general time-homogeneous diffusions by El Karoui et al~\cite{ElKarouiJeanblancShreve:98} using stochastic flows, Bergman et al~\cite{BergmanGrundyWiener:96} using pdes and Hobson~\cite{Hobson:98} using coupling.

Now return to the classical optimal stopping problem \eqref{eq:wdef}. Again, a simple coupling argument gives that if $g$ is increasing then so is $w$. Merton~\cite[Theorem 10]{Merton:73} shows that if $g$ is convex and $X$ is exponential Brownian motion then $w$ is convex. Hobson~\cite{Hobson:98}, see also Ekstr{\o}m~\cite{Ekstrom:04}, gives a coupling argument to show that if $X$ is a martingale diffusion and $g$ is convex then $w$ is convex. If we look for results which apply simultaneously across all diffusions then this is the best we can hope for (see Example~\ref{eg:nonconvex} below) although in the non-martingale case Alvarez~\cite{Alvarez:03} gives sufficient conditions for convexity which combine the payoff and the minimal decreasing $\beta$-excessive function of a given diffusion.

The first goal of this paper is to consider similar issues for $V_\theta$. If $g$ is increasing in $x$, does $V_\theta$ inherit this monotonicity property? If $g$ is convex, does $V_\theta$ inherit convexity? We give an example to show that monotonicity of $g$ is not sufficient for monotonicity of $V_\theta$, and convexity of $g$ is not sufficient for convexity of $V_\theta$, even when $X$ is a martingale diffusion.

Our first results are that if $g$ and $\theta$ are both increasing, then $V_\theta$ is increasing, and if $g$ is convex (and $X$ is a martingale) then $V_\lambda$ is convex. We give simple coupling proofs of these statements.
Our main result is more refined, and includes the above results as special cases: subject to regularity conditions, if $\theta$ and $\frac{g \theta}{\beta + \theta}$ are increasing then $V_\theta$ is increasing, and if $\frac{g \theta}{\beta + \theta}$ is convex (and $X$ is a martingale) then $V_\theta$ is convex.
Again, our proofs depend on coupling arguments. Our main technique is to show that there is a time-change $\Lambda= (\Lambda_s)_{s \geq 0}$ such that if $Y=(Y_s)_{s \geq 0}$ is given by $Y_s=X_{\Lambda_s}$ then
\begin{equation}
\label{eq:Gapparent} \E^x \left[ e^{-\beta T^\theta_1} g(X_{T^\theta_1}) \right] =
\E^x \left[ \frac{g(Y_T)\theta(Y_T)}{\beta + \theta(Y_T)} \right]
\end{equation}
where $T$ is an independent unit-rate exponential random variable.
We use this representation to show that if $\Psi:= \frac{\theta g}{\beta + \theta}$ has monotonicity (respectively convexity) properties in $x$ then so does $G_\theta(x) := \E^x \left[ e^{-\beta T^\theta_1} g(X_{T^\theta_1}) \right]$ (for convexity in $x$ we need that $X$ is a martingale). Then we deduce corresponding properties for $V^{(\infty)}_\theta$. The key role of the shape of $\Psi$ is apparent from \eqref{eq:Gapparent}.


The second goal of the paper is to consider the relationship between $V_\theta$ and $V^{(\infty)}_\theta$. Clearly $V^{(\infty)}_\theta \leq V_\theta$. We show by example that the equality may be strict. However, subject to a growth condition on $g$ and the condition that the time of the $n^{th}$ event of the Poison process increases to infinity, there is equality and $V_\theta^{(n)}$ approaches $V_\theta^{(\infty)}=V_\theta$.

The paper is structured as follows. The next section contains some simple, stylized examples, or rather counterexamples, which show in part that the questions we consider are interesting. Section~\ref{sec:formulation} gives a precise formulation of the problem, gives some first results, and explains how to change the problem for a general one-dimensional diffusion to a problem involving a diffusion in natural scale.
Section~\ref{sec:monotonicity} discusses the monotonicity and convexity of $V^{(\infty)}_\theta$. Finally,
Section~\ref{sec:Vinfty} compares $V^{(\infty)}_\theta$ to $V_\theta$ and gives conditions such that $V_\theta^{(\infty)}=V_\theta$, and hence deduces monotonicity and convexity results for $V_\theta$.

\section{Examples and counterexamples}
\label{sec:eg}

\begin{eg}\label{eg:w>V}
We might expect $\lim_{\lambda \uparrow \infty} V_\lambda(x)= w(x)$, but this is not always the case.

Let $X$ be Brownian motion on $\R$ and let $g(x)= I_{ \{ x \in \Q \} }$. Then $w(x)=1 > V_\lambda(x)=0$.

We conclude that we expect to need some conditions on $g$ in order to get reasonable results. 
\end{eg}

\begin{eg}\label{eg:nonconvex}
Let $X$ be Brownian motion with positive unit drift on $[0,\infty)$, absorbed at zero. Let $H_z$ denote the first hitting time by $X$ of $z$. Let $g(x)=x$ and let $y = \argmax \{ \frac{ze^z}{\sinh (z\sqrt{1 + 2 \beta})} \}$. If $X_0 = x$ and $dX_t = dB_t + dt$ then for $0 < x \leq y$,
\begin{equation}
\label{eq:sinh}
 w(x) = \E^x[e^{- \beta (H_0 \wedge H_y)} X_{H_0 \wedge H_y}] = y \frac{e^{(y-x)} \sinh(x \sqrt{1 + 2 \beta})}{\sinh (y \sqrt{1 + 2 \beta})}
\end{equation}
with $w(x)=x$ for $x \geq y$ (see Borodin and Salminen~\cite[3.0.5(b)]{BorodinSalminen:02} for the second equality in \eqref{eq:sinh}). It follows that $w$ is neither convex nor concave.

We conclude that unless $X$ is a martingale there is no reason to expect that convex $g$ leads to convex $w$, and {\em a fortiori} that convex $g$ leads to convex $V_\lambda$ or $V_\theta$.
\end{eg}

For the next example, and for use in other examples later in the article, for $\zeta>0$ let $\alpha^+_\zeta$ (respectively $\alpha^-_\zeta$) be the positive (respectively negative) root of $Q_\zeta(\alpha) = 0$ where
\[ Q_\zeta(\alpha) = \frac{\sigma^2}{2} \alpha(\alpha - 1) + \mu \alpha - \zeta. \]
Note that if $\zeta > \mu$ then $\alpha^+_\zeta > 1$.

\begin{eg}[Dupuis and Wang]\label{eg:DW}
Suppose $X$ is exponential Brownian motion, with drift $\mu<\beta$ and volatility $\sigma>0$. Suppose $g(x)=(x-K)^+$ and consider stopping times which are constrained to lie in the set of events times of a time-homogeneous Poisson process with rate $\lambda$.

Let $L = K ( 1 + \frac{\lambda}{(\beta + \lambda) \alpha^+_\beta  - \beta \alpha^-_{\beta + \lambda} - \lambda})$.
Then the optimal stopping time is $\tau = \inf \{ u \in \T^\lambda : X_u \geq L \}$ and
\[ V_\lambda(x) = \left\{ \begin{array}{ll}
                               (L-K) \left(\frac{x}{L} \right)^{\alpha^+_\beta}  &  0 < x \leq L  \\
                               \frac{\beta}{\beta + \lambda} (L-K) \left( \frac{x}{L} \right)^{\alpha^{-}_{\beta + \lambda}} + \frac{\lambda(x-K)}{\beta + \lambda}  & x > L . \end{array} \right. \]

In this example $V_\lambda(x) > g(x)$ on $(0,L)$ and $V_\lambda(x) < g(x)$ on $(L,\infty)$. Note that as $\lambda \uparrow \infty$, $L \uparrow M = K( \frac{1+ \alpha^+_\beta}{\alpha^+_\beta})$ and $V_\lambda(x) \uparrow w(x)$ where
\begin{equation}
\label{eq:American}
w(x) = \left\{ \begin{array}{ll}
                               (M-K) \left(\frac{x}{M} \right)^{\alpha^+_\beta}  &  0 < x \leq M  \\
                               (x-K)  & x > M.  \end{array} \right.
\end{equation}

For future reference, note that in this canonical example
\[ \E^x\left[ \sup_{s \geq t} e^{-\beta s} g(X_s) \right] \leq \E^x\left[ \sup_{s \geq t} e^{-\beta s} X_s \right] = x \frac{e^{-(\beta-\mu)t } \sigma^2}{2(\beta - \mu)} \stackrel{t \uparrow \infty}{\longrightarrow} 0. \]
\end{eg}

\begin{eg}\label{eg:nonmonotone}
Suppose $g(x)=x$ and suppose $X$ is exponential Brownian motion started at $x>0$, with volatility $\sigma$ and drift $\mu$ with $\mu < \beta$. Then $w(x)=x$
(it is always optimal to stop immediately) and $V_\lambda(x) = \rho x$ where $\rho=\frac{\lambda}{\lambda + \beta - \mu} \in (0,1)$. To see this note that it is always optimal to stop at the first event of the Poisson process and then with $T^\gamma$ denoting an exponential random variable with rate $\gamma$
\[ V_\lambda(x) = \E^x[ X_{T^\lambda} e^{- \beta T^\lambda}] = x \E[ e^{-(\beta - \mu)T^\lambda}] = x \Prob(T^\lambda < T^{\beta - \mu}) = \frac{\lambda}{\lambda + \beta - \mu} x . \]

Now suppose $\theta(x)=\infty$ for $x \leq J$ and $\theta(x)=0$ for $x>J$. Then, for $0 < x \leq J$, $V_\theta(x)=x$. For $x>J$, $V_\theta(x) = \E^x[J e^{-\beta H_J}]$. In particular, $V_\theta(x) = J( \frac{x}{J} )^{\alpha^-_\beta}$.

We conclude that monotonicity of $g$ is not sufficient for monotonicity of $V_\theta$, and that even in the martingale case $\mu = 0$, convexity of $g$ is not sufficient for convexity of $V_\theta$.
\end{eg}

\begin{eg}\label{eg:nonequality}
Suppose $X$ is standard Brownian motion absorbed at zero and started above zero. Suppose $g(x) = I_{ \{x=0 \} }$. Then $w(x) = \E^x[e^{-\beta H_0}] = e^{-\sqrt{2 \beta} x}$ on $[0,\infty)$.

Suppose $\theta(x)=x^{-2}$ on $(0,\infty)$ and $\theta(0)=1$. It can be shown that $V^{(\infty)}_\theta(x)=0$ for $x > 0$ and $V^{(\infty)}_\theta(0) = \frac{1}{1 + \beta}$. However, $V_\theta(x)=\frac{1}{1+\beta} e^{-\sqrt{2 \beta}x}$ for $x > 0$ and $V_\theta(0)= \frac{1}{1+\beta}$ so that $V^{(\infty)}_\theta < V_\theta$ on $(0,\infty)$.

We conclude that the sequence $(V^{(n)}_\theta)_{n \geq 0}$ does not always yield a limit equal to the value function $V_\theta$. In this example there are an infinite number of events of the inhomogeneous Poisson process before $X$ hits $0$ and hence $V^{(\infty)}_\theta(x) = \lim_n V^{(n)}_\theta(x) = 0$ on $(0,\infty)$. However, in calculating $V_\theta$, all these events of the Poisson process can be viewed as suboptimal as candidate stopping times. Instead the optimal stopping time is $\tau = \inf \{t \in \T^\theta: X_t = 0 \}$.
\end{eg}

\section{Problem formualation and first results}
\label{sec:formulation}

\subsection{Problem specification}
Let the stochastic process $X = (X_t)_{t \geq 0}$ be a time-homogeneous, real-valued, regular diffusion process with initial value $X_0=x$, living on a filtered probability space $\sP=(\Omega, \sF, \Prob, \F=(F_t)_{t \geq 0})$ which satisfies the usual conditions. Let $\sI \subseteq \R$ denote the state space of $X$, and suppose that any endpoints which can be reached in finite time are absorbing and are included in $\sI$. (See Section~\ref{ssec:boundary} below for further discussion about the behaviour of $X$ at endpoints of $\sI$.) We will write $\Prob^x$ to denote probabilities under the condition that $X_0=x$ (although later when we have multiple processes on the same probability space, we will also denote this dependence on the initial condition via a superscript on $X$). We suppose that $X$ solves the SDE
\begin{equation}
\label{eq:sdeX}
dX_t = a(X_t) dB_t + b(X_t) dt
\end{equation}
with initial condition $X_0=x \in \sI$, and that $a$ and $b$ are such that the solution to \eqref{eq:sdeX} is unique in law. The results of Engelbert and Schmidt~\cite{EnglebertSchmidt:84}, see Karatzas and Shreve~\cite[Section 5.5]{KaratzasShreve:91}, show that a sufficient condition is that $1/a^2$ and $b/a^2$ are locally integrable.

Let $g: \sI \mapsto \R_+$ be a non-negative (measurable) payoff function and let $\beta$ be a strictly positive discount factor. In principle our results can be extended to the case of state-dependent discount factors, but the focus in this paper is on state-dependent arrival rates for stopping opportunities and we will suppose that the discount factor is constant.

The value function $w$ of the classical discounted optimal stopping problem is defined as
\begin{equation}
\label{eq:wdef2}
w(x) = \sup_{\tau \in \sT([0,\infty))} \E^x [ e^{-\beta \tau} g(X_\tau)]
\end{equation}
where $\sT(\T)$ is the set of all $\T$-valued stopping times.
\begin{sass}
\label{sass:1} The coefficients of the SDE for $X$ are such that $a>0$ and $1/a^2$ and $b/a^2$ are locally integrable, so that $X$ is unique in law.
Further, $g \geq 0$ satisfies suitable growth conditions, so that the problem for $w$ in \eqref{eq:wdef2} is well-posed.
\end{sass}

Now consider a Poisson optimal stopping problem in which stopping can only occur at the event times $ \T^\lambda = \{ T^\lambda_n \}_{n \geq 1}$ of an independent Poisson process of rate $\lambda$. (We assume that the probability space is rich enough to carry a Poisson process which is independent of $X$, and to carry any other random variables which we wish to define.) The value function is now given by
\begin{equation}
\label{eq:hdef2}
 V_\lambda(x) = \sup_{\tau \in \sT(\T^\lambda)} \E^x [ e^{-\beta \tau} g(X_\tau)]
\end{equation}
where $\T^\lambda$ is the set of event times of a Poisson process rate $\lambda$.
We expect that as $\lambda$ increases then $\lim_{\lambda \uparrow \infty} V_\lambda(x) = w(x)$, at least if $g$ is lower semi-continuous. As we saw in Example~\ref{eg:w>V}, in general equality in the limit may fail.

Let $H_\lambda$ be the value of the Poisson optimal stopping problem, conditional on there being an event of the Poisson process at time 0. Then we have
\begin{equation}
\label{eq:Hdef}
H_\lambda(x) = \sup_{\tau \in \sT(\T^\lambda \cup \{ 0 \} )} \E^x [ e^{-\beta \tau} g(X_\tau)] = \max \{ g(x), V_\lambda (x) \} .
\end{equation}
Further, by conditioning on the first event time of the Poisson process we have the representation
$V_\lambda(x) = \E^x \left[ \int_0^\infty dt \; \lambda e^{-\lambda t} e^{- \beta t} H_\lambda(X_t) \right] $.
Substituting \eqref{eq:Hdef} into 
this last equality gives an expression for $V_\lambda$ in feedback form:
\begin{equation}
\label{eq:hdeffeedback}
 V_\lambda(x) = \E^x \left[ \int_0^\infty dt \; \lambda e^{-\lambda t} e^{- \beta t} \max \{ g(X_t), V_\lambda (X_t) \} \right] .
\end{equation}
Based on this identity we expect that $V_\lambda$ will solve the ode
\[ \sL V - (\beta + \lambda) V + \lambda(g \vee V) = 0 \]
where $\sL$ is the generator of $X$. Dupuis and Wang~\cite{DupuisWang:02} discus the solution of \eqref{eq:hdef2} and write down expressions for $V_\lambda$ and the continuation region in the case where $X$ is exponential Brownian motion and $g$ is a call payoff, see Example~\ref{eg:DW}. Lempa~\cite{Lempa:12} extends these results to general diffusions.

Let $\theta : \sI \mapsto [0,\infty)$ be a measurable function such that, to avoid trivialities, $\int_{\sI} \theta(x) dx > 0$. We consider $\theta$ to be the stochastic rate function of a state-dependent Poisson process $N^\theta = (N^\theta_t)_{t \geq 0}$ so that, conditional on the path of the diffusion $X$, the probability that there are no events of the Poisson process in an interval $[s,t)$ is $\exp( - \int_{[s,t)} \theta(X_u) du)$. Let $\T^\theta$ denote the event times of this Poisson process and let $\sT(\T^\theta)$ be the set of stopping times constrained to take values in the event times of $N^\theta$.

Let $T^\theta_1$ be the first event time. We can write $\{ T^\theta_1, T^\theta_2, \ldots T^\theta_n \}$ for the first $n$ events, but note that there may be countably infinitely many events in finite time. As a result, we cannot always write the set of event times as $\{ T^\theta_n \}_{n \geq 1}$, at least not if we insist on $T^\theta_i < T^\theta_j$ for $i<j$.

We wish to consider the properties of
\[ G_\theta(x) = \E^x[e^{- \beta T^\theta_1} g(X_{T^\theta_1})] \]
and especially
\[ V_\theta(x) = \sup_{\tau \in \sT(\T^\theta)} \E^x[e^{- \beta \tau} g(X_{\tau})] . \]
Where the arrival rate of the Poisson process is constant and equal to $\lambda$ we write $G_\lambda$ instead of $G_\theta$.

\subsection{First results}
\label{ssec:firstresults}
In this section we give some simple proofs of monotonicity and convexity of $G_\theta$ and $V_\theta$ which can be obtained by extending proofs of monotonicity and convexity for $w$ from the literature
(see \cite{BergmanGrundyWiener:96,CoxRoss:76,ElKarouiJeanblancShreve:98,Ekstrom:04,HendersonSunWhalley:14,Hobson:98,Lindvall:92,Merton:73}). In Section~\ref{sec:monotonicity} we will give stronger results using a different coupling which is specific to the Poisson optimal stopping problem.

Under Standing Assumption~\ref{sass:1} the diffusion $X$ is unique in law, and the optimal stopping problem corresponding to $w$ is well-posed. Then $V_\theta$ is finite.
\begin{thm}
\label{thm:simpleincreasing}
Suppose $g$ and $\theta$ are increasing in $x$. Then $V_\theta$ is increasing in $x$.
\end{thm}

\begin{proof}

Suppose $X$ solves
\begin{equation}
\label{eq:Xsde}
dX_t = a(X_t) dB_t + b(X_t) dt
\end{equation}
Fix $x<y$. Let $X^x$ and $X^y$ denote solutions of \eqref{eq:Xsde} where the superscript indicates the initial value e.g. $X^x_0=x$. We construct a coupling such that $X^x \leq X^y$ pathwise.

Let $\bar{X}^x$ solve $d \bar{X}^x_s = a(\bar{X}_s) d \bar{B}^x_s + b(\bar{X}_s) dt$ subject to $\bar{X}^x_0 = x$ and let $\bar{X}^y$ solve $d \bar{X}^y_s = a(\bar{X}^y_s)d \bar{B}^y_s + b(\bar{X}^y_s)ds$ subject to $\bar{X}^y_0 = y$, where the Brownian motions $\bar{B}^x$ and $\bar{B}^y$ are independent. Let $\sigma = \inf \{ u : \bar{X}^x_u = \bar{X}^y_u \}$, let $\tilde{X}^x_s = \bar{X}^x_s$ and let $\tilde{X}^y_s = \bar{X}^y_s$ on $s \leq \sigma$ and $\tilde{X}^y_s = \bar{X}^x_s$ on $s>\sigma$. Then, by the Strong Markov property and uniqueness in law, $\bar{X}^y$ and $\tilde{X}^y$ are identical in law. Moreover, $\tilde{X}^x_s \leq \tilde{X}^y_s$ by construction. This is the Doeblin coupling, Lindvall~\cite[Section II.2]{Lindvall:92}.
It follows that $\E[\psi(\bar{X}^x_s)] = \E[\psi(\tilde{X}^x_s)] \leq \E[\psi(\tilde{X}^y_s)]=\E[\psi(\bar{X}^y_s)]$ for any non-negative, increasing function $\psi$ and any $s$.

Suppose that $\theta$ is constant (in which case we write $\lambda$).
Then, since $g$ is increasing, for the coupled processes $(\tilde{X}^x,\tilde{X}^y)$ and for any $\tau$ we have $e^{-\beta \tau}g(\tilde{X}^x_\tau) < e^{-\beta \tau} g(\tilde{X}^y_\tau)$. Moreover, for $\tau \in \sT(\T^\lambda)$,
\[\E[e^{-\beta \tau}g(\tilde{X}^x_\tau)] \leq \E[e^{-\beta \tau} g(\tilde{X}^y_{\tau})] \leq \sup_{\xi \in \sT(\T^\lambda)} \E[e^{-\beta \xi} g(\tilde{X}^y_{\xi})] = V_\lambda(y). \]
Taking a supremum over $\tau \in \sT(\T^\lambda)$ gives that $V_\lambda(x) \leq V_\lambda(y)$ and hence that $V_\lambda$ is increasing in $x$.

Now we consider the corresponding result for increasing rate functions $\theta$.
By the previous analysis, without loss of generality we may assume that $X^x_s \leq X^y_s$ for all $s \geq 0$.

Let $N^\gamma = (N^{\gamma}_t)_{t \geq 0}$ be a Poisson process with stochastic rate function $\gamma = (\gamma_t)_{t \geq 0}$.

There are two natural ways to think of $N^\gamma=(N^\gamma_t)_{t \geq 0}$ and therefore (at least) two natural ways to couple inhomogeneous Poisson processes with different rates.

First, if $\bar{N}$ is a unit-rate Poisson counting process, then we can define $N^\gamma=(N^\gamma_t)_{t \geq 0}$ by $N^\gamma_t = \bar{N}_{\int_0^t \gamma_s ds}$. Then, given a pair of Poisson processes $N^\gamma$ and $N^\xi$ we can couple them by writing $N^\gamma_t = \bar{N}_{\int_0^t \gamma_s ds}$ and $N^\xi_t = \bar{N}_{\int_0^t \xi_s ds}$. If $\int_0^t \gamma_s ds \geq \int_0^t \xi_s ds$ for all $t$ then $N^\gamma_t \geq N^\xi_t$ for all $t$.

Second, we can consider $N^\gamma$ as the counting process derived from a homogeneous space-time Poisson process $N^{\R^2_+}$ in which there is an event of $N^\gamma$ in $[s,t)$ if and only if there is an event of $N^{\R^2_+}$ in $\{(u,z): s \leq u < t, z \leq \gamma_u \}$. See Figure~\ref{fig:1}. Here, $N^{\R^2_+}$ is a Poisson process in the first quadrant of the plane for which the number of points in a set $A \subseteq \R^2_+$ is a Poisson random variable with mean the area of $A$.

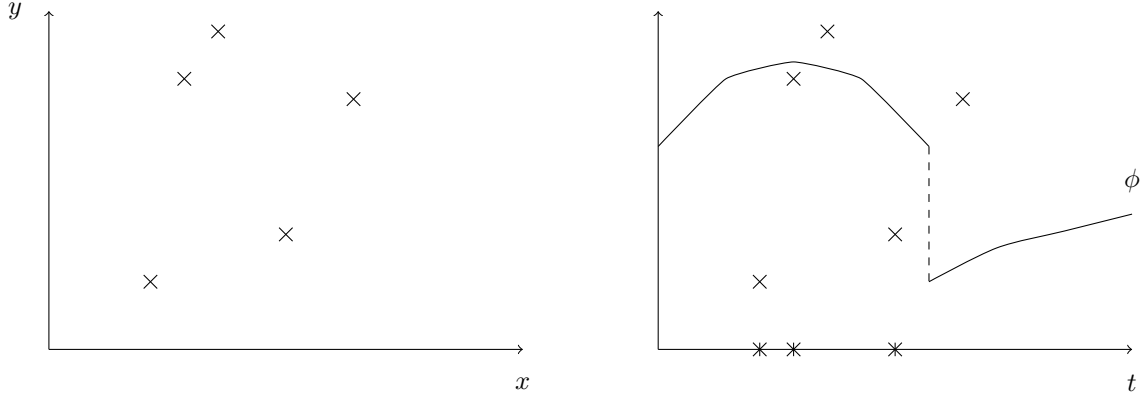
\begin{figure}[H]
\centering
\begin{tikzpicture}[scale=0.9]

 \draw[->] (0,0) -- (0,5) ;
 \draw[->] (0,0) -- (7,0) ;

 \draw[->] (9,0) -- (16,0) ;
 \draw[->] (9,0) -- (9,5) ;

 \draw[-] (1.9,3.9) -- (2.1,4.1) ;
 \draw[-] (2.1,3.9) -- (1.9,4.1) ;
 \draw[-] (10.9,3.9) -- (11.1,4.1) ;
 \draw[-] (11.1,3.9) -- (10.9,4.1) ;
 \draw[-] (10.9,-0.1) -- (11.1,0.1) ;
 \draw[-] (11.1,-0.1) -- (10.9,0.1);
 \draw[-] (11.0,-0.1) -- (11.0,0.1);

 \draw[-] (1.4,0.9) -- (1.6,1.1) ;
 \draw[-] (1.6,0.9) -- (1.4,1.1) ;
 \draw[-] (10.4,0.9) -- (10.6,1.1) ;
 \draw[-] (10.6,0.9) -- (10.4,1.1) ;
 \draw[-] (10.4,-0.1) -- (10.6,0.1) ;
 \draw[-] (10.6,-0.1) -- (10.4,0.1);
 \draw[-] (10.5,-0.1) -- (10.5,0.1);

 \draw[-] (3.4,1.6) -- (3.6,1.8) ;
 \draw[-] (3.6,1.6) -- (3.4,1.8) ;
 \draw[-] (12.4,1.6) -- (12.6,1.8) ;
 \draw[-] (12.6,1.6) -- (12.4,1.8) ;
 \draw[-] (12.4,-0.1) -- (12.6,0.1) ;
 \draw[-] (12.6,-0.1) -- (12.4,0.1);
 \draw[-] (12.5,-0.1) -- (12.5,0.1);

 \draw[-] (2.4,4.6) -- (2.6,4.8) ;
 \draw[-] (2.6,4.6) -- (2.4,4.8) ;
 \draw[-] (11.4,4.6) -- (11.6,4.8) ;
 \draw[-] (11.6,4.6) -- (11.4,4.8) ;

 \draw[-] (4.4,3.6) -- (4.6,3.8) ;
 \draw[-] (4.6,3.6) -- (4.4,3.8) ;
 \draw[-] (13.4,3.6) -- (13.6,3.8) ;
 \draw[-] (13.6,3.6) -- (13.4,3.8) ;

 \draw plot [smooth, tension=0.25] coordinates { (9,3) (10,4) (11,4.25) (12,4) (13,3) } ;
 \draw[dashed] (13,3) -- (13,1) ;
 \draw plot [smooth, tension=0.5] coordinates { (13,1) (14,1.5) (15,1.75) (16,2) };
 \node (n1) at (7,-0.5) {$x$};
 \node (n1) at (-0.5,5) {$y$};
 \node (n1) at (16,-0.5){$t$};
 \node (n1) at (16,2.5) {$\phi$};

\end{tikzpicture}
\caption{The left figure shows events of the unit rate Poisson process on $\R^2_+$. The right figure how those events become events of a time-inhomogeneous Poisson process on $\R_+$ of rate $\phi$: an event at $(x,y)$ becomes an event at $t=x$ if $y \leq \phi(x)$. }
\label{fig:1}
\end{figure}

We take the second approach. 
Since $\theta$ is increasing (and we have coupled $X^x$ and $X^y$ so that $X^x_t \leq X^y_t$ for all $t$) we have a set inclusion of the event times for the Poisson process with rate $\theta(X^x_t)_{t \geq 0}$ within the set event times for the Process with rate $\theta(X^y_t)_{t \geq 0}$:
\[ \T^{(\theta(X^x_t))_{t \geq 0}} = \{ u : (u,z) \in N^{\R_+^2}, z \leq \theta(X^x_u) \} \subseteq \{ u : (u,z) \in N^{\R_+^2}, z \leq \theta(X^y_u) \} = \T^{(\theta(X^y_t))_{t \geq 0}}. \]

In particular, any candidate stopping time for the process started at $x$ is also a candidate stopping time for the process started at $y$. Then
\[ 
\sup_{\tau \in \sT(\T^{(\theta(X^x_t))_{t \geq 0}})} \E[e^{-\beta \tau} g(X^x_\tau)] \leq \sup_{\tau \in \sT(\T^{(\theta(X^x_t))_{t \geq 0}})} \E[e^{-\beta \tau} g(X^y_\tau)] \leq \sup_{\tau \in \sT(\T^{(\theta(X^y_t))_{t \geq 0}})} \E[e^{-\beta \tau} g(X^y_\tau)] 
 \]
where the first inequality comes from $X^x_{\cdot} \leq X^y_{\cdot}$ and the second from the inclusion $\sT(\T^{(\theta(X^x_t))_{t \geq 0}}) \subseteq \sT(\T^{(\theta(X^y_t))_{t \geq 0}})$.
\end{proof}

\begin{thm}
\label{thm:EBM}
Suppose $X$ is exponential Brownian motion. Suppose $g$ is convex. Then $\E^x[e^{-\beta t} g(X_t)]$, $G_\lambda(x) = \E[ e^{-\beta T^\lambda_1} g(X_{T^\lambda_1})]$ and $V_\lambda(x)$ are convex in $x$.
\end{thm}

\begin{proof}
This result extends a result of Merton~\cite[Theorem 10]{Merton:73} from convexity of $w$ in $x$ to convexity of $V_\lambda$.

Suppose $dX_t = \sigma X_t dB_t + \mu X_t dt$. Then there is a coupling such that $X^x$ has representation
$X^x_t = x Z_t$ where $Z_t= e^{\sigma B_t + (\mu - \frac{1}{2} \sigma^2)t}$ is independent of $x$. Then for $x<y$ and $\zeta \in (0,1)$,
\begin{equation}
g(X^{\zeta x + (1-\zeta) y}_t)  =  g( \zeta x Z_t + (1 - \zeta)yZ_t) \leq \zeta g(xZ_t) + (1 - \zeta)g(yZ_t) = \zeta g(X^x_t) + (1 - \zeta)g(X^y_t).
\end{equation}
It follows that for any stopping time $\tau$ we have $g(X^{\zeta x + (1-\zeta) y}_\tau) \leq \zeta g(X^x_\tau) + (1 - \zeta)g(X^y_\tau)$ and then
\begin{equation}
\label{eq:thEBM}
 \E[e^{-\beta \tau} g(X^{\zeta x + (1-\zeta) y}_\tau)] \leq \zeta \E[e^{-\beta \tau}g(X^x_\tau)] + (1 - \zeta)\E[e^{-\beta \tau}g(X^y_\tau)].
\end{equation}
Taking $\tau = T^\lambda_1$ we get that $G_\lambda$ is convex. Moreover, taking a pair of supremums over $\tau \in \sT(\T^\lambda)$ on the right-hand-side of \eqref{eq:thEBM},
\[  \E[e^{-\beta \tau} g(X^{\zeta x + (1-\zeta) y}_\tau)] \leq \zeta V_\lambda(x) + (1 - \zeta)V_\lambda(y). \]
Now, taking a supremum over $\tau \in \sT(\T^\lambda)$ on the left-hand-side we obtain $V_\lambda(\zeta x + (1 - \zeta)y) \leq \zeta V_\lambda(x) + (1 - \zeta)V_\lambda(y)$.
\end{proof}
\begin{rem}
A similar proof applies to the case where $X$ is Brownian motion with drift and we deduce that if $X^x_t = x + aB_t + bt$ and $g$ is convex then $G_\lambda(x)$ and $V_\lambda(x)$ are convex in $x$.
\end{rem}



\begin{thm}
\label{thm:hobson}
Suppose $X$ is a martingale diffusion. Suppose $g$ is convex. Then $\E^x[ e^{- \beta t} g(X_t)]$ and $G_\lambda(x)$ are convex in $x$.
\end{thm}

\begin{proof}
This result extends Hobson~\cite[Theorem 3.1]{Hobson:98}.

For $x < y < z$ define a triple of processes $(X,Y,Z)$ via
\[ dX_t = a(X_t) dB^X_t + b(X_t) dt \hspace{20mm} X_0 = x, \]
and similarly $dY_t = a(Y_t) dB^Y_t + b(Y_t) dt$ subject to $Y_0=y$ and $dZ_t = a(Z_t) dB^Z_t + b(Z_t) dt$ subject to $Z_0=z$. (Here we use the more economical notation $(X,Y,Z)$ where normally we might write $(X^x, X^y, X^z)$.)

Couple the processes by making the three driving Brownian motions independent.
Let $H^{xy} = \inf \{ u : X_u = Y_u \}$ and $H^{yz}= \inf \{ u : Y_u = Z_u \}$. Fix $t > 0$ and let $\sigma = H^{xy} \wedge H^{yz} \wedge t$. Then, by symmetry, on $\sigma = H^{xy}$,
\[ (Z_t - X_t)g(Y_t) \stackrel{\sL}{=} (Z_t - Y_t)g(X_t)  \hspace{20mm} Y_t g(Z_t) \stackrel{\sL}{=} X_t g(Z_t) \]
so that
\begin{equation}
\label{eq:XYZ1}
\E[ (Z_t - X_t) g(Y_t)I_{ \{\sigma = H^{xy} \} }] = \E[ (Z_t - Y_t)g(X_t) I_{ \{  \sigma = H^{xy} \} }] + \E[ (Y_t - X_t)g(Z_t) I_{ \{ \sigma = H^{xy} \} }]. \end{equation}

Similarly, 
\begin{equation}
\label{eq:XYZ2}
 \E[ (Z_t - X_t) g(Y_t) I_{ \{\sigma = H^{yz} \} } ] = \E[ (Z_t - Y_t)g(X_t)I_{ \{\sigma = H^{yz}\}}] + \E[ (Y_t - X_t)g(Z_t)I_{\{\sigma = H^{yz}\}}].
\end{equation}

Finally, on $H^{xy} \wedge H^{yz}> t$ we have $\sigma = t$, $X_t < Y_t < Z_t$ and by convexity of $g$
\[ (Z_t - X_t) g(Y_t) I_{ \{ \sigma < H^{xy}\wedge H^{xz} \} }  \leq (Z_t - Y_t)g(X_t)I_{ \{ \sigma < H^{xy}\wedge H^{xz} \} } + (Y_t - X_t)g(Z_t) I_{ \{ \sigma < H^{xy}\wedge H^{xz} \} }. \]
Taking expectations, adding the result to \eqref{eq:XYZ1} and \eqref{eq:XYZ2}, and multiplying by $e^{-\beta t}$ we obtain
\[ \E[ (Z_t - X_t) e^{-\beta t}g(Y_t)] \leq \E[ (Z_t - Y_t) e^{-\beta t}g(X_t)] + \E[ (Y_t - X_t)e^{-\beta t}g(Z_t) ]. \]
Using the fact that $X$, $Y$ and $Z$ are independent
we conclude that
\[ (z-x) \E[e^{-\beta t}g(Y_t)] \leq(z - y) \E[e^{-\beta t}g(X_t)] +(y - x)\E[e^{-\beta t}g(Z_t) ] \]
and that $\E[e^{-\beta t} g(X^x_t)]$ is convex in $x$.

Since $G_\lambda(x) = \int_0^\infty \lambda e^{-\lambda t} \E^x[e^{-\beta t} g(X^x_t)] dt$ the convexity property is also inherited by $G_\lambda$.


\end{proof}

\begin{rem}
\label{rem:concave}
The same proof shows that if $g$ is concave, then $G_\lambda$ is concave.
\end{rem}

We close this section with two other results which will be useful in later sections.

\begin{prop}
\label{prop:V=G}
If $G_\theta \leq g$ then $V_\theta^{(\infty)} = V_\theta^{(n)} = V_\theta^{(1)} = G_\theta$.
\end{prop}

\begin{proof}
Suppose $V^{(k)}_\theta = G_\theta \leq g$. This is true for $k=1$ by hypothesis. Then
\[ V^{(k+1)}_\theta(x) = \E^x[ e^{-\beta T^\theta_1} \{ g(X_{T^\theta_1}) \vee V^{(k)}_\theta(X_{T^\theta_1}) \}] = \E^x[  e^{-\beta T^\theta_1} g(X_{T^\theta_1})] = G_\theta(x) \leq g(x) \]
and the result follows by induction.
\end{proof}

\begin{prop}
\label{prop:cuteconvexity}
Let $Y$ be a regular martingale diffusion with state space $\sI$ and let $T$ be an independent exponential random variable.
Suppose $c: \sI \rightarrow [0,\infty)$ is bounded on compact sub-intervals of $\sI$ and is such that $\E^y[c(Y_T)] \geq c(y)$.

Let $C(y) = \E^y[ c(Y_T) ]$. Then $C$ is convex.
\end{prop}

Note that convexity of $c$ is a sufficient but not necessary condition for $\E^y[c(Y_T)] \geq c(y)$.
\begin{proof}
Let $\{T_1, T_2, \ldots \}$ be the event times of a Poisson process, let $T_0=0$ and let $\{ S_k = T_k - T_{k-1} \}_{k \geq 1}$ be the inter-arrival times.

Fix $x,y,z \in \sI$ with $x<y<z$. Let $Y_0=y$ and for $w \in \{x,z \}$ define $H^t_w = \inf \{u > t: Y_u = w \}$. We have $C(y) = \E^y[c(Y_{T_1})]$ and then
\[ C(y) = \E^y\left[c(Y_{T_1}) I_{ \{ H_x \leq H_z \wedge T_1 \} }\right] + \E^y\left[c(Y_{T_1}) I_{ \{ H_z \leq H_x \wedge T_1 \} }\right] + \E^y\left[c(Y_{T_1}) I_{ \{ T_1 < H_x \wedge H_z \} }\right]. \]
By the Strong Markov property of $Y$ and the fact that $T_1$ is memoryless we have
\begin{eqnarray*}
\E^y\left[c(Y_{T_1}) I_{ \{ H_x \leq H_z \wedge T_1 \} }\right] & = & \E^y\left[ \E^y [ c(Y_{T_1}) I_{ \{ H_x \leq H_z \wedge T_1 \} }| \sF_{H_x \wedge H_y \wedge T_1} ] \right] \\
& = & \E^y \left[ \E^x[ c(Y_{T_1}) ] I_{ \{ H_x \leq H_z \wedge T_1 \} } \right] \\
&=& C(x) \Prob^y(H_x \leq H_z \wedge T_1).
\end{eqnarray*}
Similarly, $\E^y\left[c(Y_{T_1}) I_{ \{ H_z \leq H_x \wedge T_1 \} }\right] = C(z) \Prob^y(H_z \leq H_x \wedge T_1)$.

Suppose inductively that
\begin{equation}
\label{eq:induction} C(y) \leq C(x) \Prob^y(H_x \leq H_z \wedge T_k) + C(z) \Prob^y(H_z \leq H_x \wedge T_k) +  \E^y\left[ c(Y_{T_k})I_{ \{ T_k < H_x \wedge H_z \} } \right].
\end{equation}
We have shown this is true for $k=1$. Let $Y^{T_k}$ be given by $Y^{T_k}_t = Y_{T_k+t}$. Then, on $T_k < H_x \wedge H_z$, and writing $S$ for $S_{k+1}$,
\begin{eqnarray*}
c(Y_{T_k}) & \leq & \E^{Y_{T_k}}\left[ \left. c(Y_{T_{k+1}}) \right| \sF_{T_k} \right] \\
& = &  \E^{Y_{T_k}} \left[c(Y^{T_k}_S) I_{ \{ H_x \leq H_z \wedge (T_k+S) \} }\right] + \E^{Y_{T_k}}\left[c(Y^{T_k}_S) I_{ \{ H_z \leq H_x \wedge (T_k+S) \} }\right] \\
&& \hspace{70mm} + \E^{Y_{T_k}}\left[c(Y^{T_k}_S) I_{ \{ T_k+S < H_x \wedge H_z \} }\right] \\
& = & C(x) \Prob^{Y_{T_k}} (H_x \leq H_z \wedge (T_k+S)) + C(z)  \Prob^{Y_{T_k}} (H_x \leq H_z \wedge (T_k+S)) \\
&& \hspace{70mm}+ \E^{Y_{T_k}}\left[c(Y^{T_k}_S) I_{ \{ T_k+S < H_x \wedge H_z \} }\right] .
\end{eqnarray*}
It follows that
\begin{eqnarray*}
\E^y\left[ c(Y_{T_k})I_{ \{ T_k < H_x \wedge H_z \} } \right]
& \leq & C(x) \Prob^y(T_k < H_x \leq H_z \wedge T_{k+1}) \\
& & \hspace{8mm}+ C(z) \Prob^y(T_k < H_z \leq H_x \wedge T_{k+1}) +  \E^y\left[ c(Y_{T_{k+1}})I_{ \{ T_{k+1} < H_x \wedge H_z \} } \right].
\end{eqnarray*}
Substituting this inequality into \eqref{eq:induction} we get the equivalent statement for $k+1$. Hence we know that \eqref{eq:induction} holds for all $k\geq1$. Letting $k \uparrow \infty$, and using the fact that $Y$ is regular and $c$ is bounded on $[x,z]$ we get
\[C(y) \leq C(x) \Prob^y(H_x \leq H_z) + C(z) \Prob^y(H_z \leq H_x). \]
Then, using the martingale property of $Y$ we get $C(y) \leq C(x) \frac{z-y}{z-x} + C(z)\frac{y-x}{z-x}$ and $C$ is convex.
\end{proof}

\begin{rem}
The argument extends without change to cover the case where the unit-rate exponential $T$ is replaced by the first event time $T^\theta_1$ of a Poisson process with rate $\theta = \{ \theta(Y_t) \}_{t \geq 0}$, provided $T^\theta_1$ is almost surely finite. An alternative strategy for a proof is to use the fact that we expect $C$ to solve $\sL^Y C  = \theta (C-c)$, where $\sL^Y$ is the generator of $Y$. Then, since $\sL$ has no first order derivative, if $C \geq c$ everywhere, then $C$ is convex.
\end{rem}

\begin{eg}
Let $B$ be Brownian motion. For $\phi \geq 0$ set $h_\phi(x) = |x| + \phi \left\{ \frac{|1-x| + |1+x|}{2} - |x| \right\}$. Then $h_\phi$ is symmetric about zero, and piecewise linear with kinks at $0$ and $\pm 1$. Moreover, $h_\phi(0)=\phi$ and $h_\phi(x)= |x|$ for $|x| \geq 1$. Note that $h_\phi$ is convex if and only if $\phi \leq 1$.

Let $T_\lambda$ be an exponential of rate $\lambda>0$ and let $\xi = \sqrt{2 \lambda}$. Set $H_\phi(x) = \E^x[h_\phi(B_{T_\lambda})]$. Then, with $L^{B,y}_t$ denoting the local time of $B$ at $y$ by time $t$,
\begin{eqnarray*} H_\phi(x) & = & h_\phi(x) + \frac{\phi}{2} \E^x[L^{B,1}_{T_\lambda}] + \frac{\phi}{2} \E^x[L^{B,-1}_{T_\lambda}] + (1-\phi) \E^x[L^{B,0}_{T_\lambda}] \\
& = & h_\phi(x) + \frac{\phi}{2} \frac{e^{-\xi |1-x|}}{\xi} + \frac{\phi}{2} \frac{e^{-\xi |1+x|}}{\xi} + (1-\phi) \frac{e^{-\xi |x|}}{\xi}.
\end{eqnarray*}
Then, for $x \in (-1,1)$,
\[ H_\phi''(x) = \xi^2(H_\phi(x)- h_\phi(x)) = \xi \left[ \phi e^{-\xi} \cosh (\xi x) + (1-\phi) e^{-\xi |x|} \right], \]
and for $|x| \geq 1$,
\[ H_\phi''(x) = \xi^2(H_\phi(x)- h_\phi(x)) = \xi e^{-\xi |x|} \left[ \phi \cosh \xi + (1-\phi) \right]. \]
Then $H_\phi$ is convex everywhere if and only if $H_\phi \geq h_\phi$ everywhere, if and only if $\phi \leq \frac{1}{1 - e^{-\xi}}$.
In particular, if $1<\phi \leq  \frac{1}{1 - e^{-\xi}}$ then $H_\phi$ is convex, even though the payoff function $h_\phi$ is not.
\end{eg}

\subsection{Reduction of the problem to a problem in natural scale}
\label{ssec:natscale}
Recall that our assumption is that $X$ is a regular diffusion with state space $\sI$ which solves the SDE $dX_t = a(X_t)dB_t + b(X_t) dt$. Moreover, Standing Assumption~\ref{sass:1} gives that $b/a^2$ is locally integrable. Then we can define $s:\sI \rightarrow \R$ by
\( s'(x) = \exp ( - \int^x \frac{b(z)}{a(z)^2} dz ) \)
and if we set $M_t=s(X_t)$ then $M=(M_t)_{t \geq 0}$ solves $dM_t = \eta(M_t) dB_t$ where $\eta = (a s') \circ s^{-1}$. The key point is that $M$ is a (local) martingale. Moreover $M$ is a regular diffusion with state space $\sI_M = s(\sI)$. The increasing, invertible function $s$ is called the scale function and $M$ is said to be in natural scale (Rogers and Williams~\cite[V.46]{RogersWilliams:01}). Note that $s$ is only determined up to a linear transformation, so we may choose constants to make $\sI_M$ have a convenient form. 

Let $\hat{g}= g \circ s^{-1}$ and $\hat{\theta} = \theta \circ s^{-1}$. Then $e^{-\beta \tau} g(X_\tau) = e^{-\beta \tau} \hat{g}(M_\tau)$ and ${\theta}(X_t) = \hat{\theta}(M_t)$ so that the inhomogeneous Poisson process with rate $(\theta(X_t))_{t \geq 0}$ can be identified with the inhomogeneous Poisson process with rate $(\hat{\theta}(M_t))_{t \geq 0}$. Finally,
\[ \hat{V}_{\hat{\theta}}(m): = \sup_{\tau \in \sT(\T^{\hat{\theta}})} \E^{M_0=m} \left[e^{-\beta \tau} \hat{g}(M_\tau) \right]
     = \sup_{\tau \in \sT(\T^{{\theta}})} \E^{X_0=s^{-1}(m)} \left[e^{-\beta \tau} {g}(X_\tau) \right] = V_\theta (s^{-1}(m))  \]
so that $\hat{V}_{\hat{\theta}} = V_\theta \circ s^{-1}$.

Since $s^{-1}$ is increasing we conclude that proving that $V_\theta$ is increasing is equivalent to proving that $\hat{V}_{\hat{\theta}}$ is increasing. Hence we may restrict attention to diffusions in natural scale.

When we turn to problems concerning convexity, then, recall Example~\ref{eg:nonconvex}, we only expect general convexity results for $V_\theta$ in cases where the diffusion $X$ is already in natural scale.

\subsection{Boundary behaviour}
\label{ssec:boundary}
Suppose $M$ is a regular diffusion in natural scale with state space $\sI_M$ with endpoints $\hat{\ell}$ and $\hat{r}$ with $\hat{\ell}<\hat{r}$. Suppose $M$ solves $dM_t = \eta(M_t) dB_t$ where $1/\eta^2$ is locally integrable.

Suppose $\hat{e} \in \{ \hat{\ell}, \hat{r} \}$ is finite. If $M$ can reach $\hat{e}$ in finite time then we say $\hat{e}$ is accessible (see Rogers and Williams~\cite[Section V.47]{RogersWilliams:01} or Revuz and Yor~\cite[Section VII.3]{RevuzYor:99} for terminology). If $\hat{e}$ is accessible then we assume that $M$ is absorbed at $\hat{e}$. The necessary and sufficient condition that $\hat{e}$ can be reached in finite time is $I_\eta(\hat{e}) < \infty$ where $I_\eta(\hat{e})=\int_{\hat{e}} |m-\hat{e}| \frac{dm}{\eta(m)^2}$. Otherwise, if $I_\eta(\hat{e})=\infty$, then $\hat{e}$ cannot be reached in finite time and we say $\hat{e}$ is a natural boundary. If $\hat{e}$ is a finite, accessible endpoint then $\hat{e} \in \sI_M$; otherwise, if $\hat{e}$ is natural then $\hat{e} \notin \sI_M$.

Now suppose $\hat{e} \in \{ \hat{\ell}, \hat{r} \}$ is infinite. If, for $y \in (\hat{\ell},\hat{r})$ we have $\lim_{x \rightarrow \hat{e}} \Prob^x(H_y < \infty) > 0$ (or equivalently $\lim_{x \rightarrow \hat{e}} \E^x[e^{- \gamma H_y}] > 0$ for each $\gamma>0$) then $\hat{e}$ is an entrance boundary. The condition that $\hat{e}$ is an entrance boundary is $J_\eta(\hat{e})<\infty$ where $J_\eta(\hat{e})=\int_{\hat{e}} \frac{dm}{\eta(m)^2}$. Otherwise $\hat{e}$ is a natural boundary and $\hat{e} \notin \sI_M$. It is not possible for $M$ to explode to an infinite boundary point in finite time.

Suppose $X$ solving \eqref{eq:Xsde} is a time-homogeneous regular diffusion, not in natural scale, on a state space $\sI$ with endpoints $\ell$ and $r$. We classify the boundary points of $X$ by using the classification of the corresponding boundary points for $M=s(X)$. In particular, for $e \in \{ \ell,r \}$, if $|s(e)|<\infty$ and $\int_e \frac{|s(x)-s(e)|}{s'(x) a(x)^2} dx < \infty$ then $e$ can be reached in finite time, and we take $e$ to be absorbing. If $|s(e)| < \infty$ and $\int_e \frac{|s(x)-s(e)|}{s'(x) a(x)^2} dx = \infty$ or if $|s(e)| = \infty$ and $\int_e \frac{1}{s'(x) a(x)^2} dx < \infty$ then $e$ is natural.

\begin{sass}
\label{sass:2}
Boundary points are either natural, or if they can be reached in finite time, they are absorbing.
\end{sass}

\section{Monotonicity and convexity of $V^{(\infty)}_\theta$.}\label{sec:monotonicity}

Consider the solution $V^{(n)}_\theta$ of the Poisson optimal stopping problem, under the restriction that stopping must occur at one of the first $n$ events of the Poisson process $N^\theta$. We have
\begin{equation}
\label{eq:Vndef}
V_\theta^{(n)}(x) = \sup_{\tau \in \sT{(\{T^\theta_1, T^\theta_2 \ldots T^\theta_n}\})} \E^x \left[ e^{-\beta \tau} g(X_\tau) \right] .
\end{equation}
Set $V_\theta^{(0)}(x) = 0$. Then $V_\theta^{(1)} = G_\theta$ and  $V^{(n)}_\theta(x) =   \E^x \left[ e^{-\beta T^\theta_1} \max \left\{ g(X_{T^\theta_1}) , V^{(n-1)}_\theta(X_{T^\theta_1}) \right\} \right]$.

Lange et al~\cite{LangeRalphStore:19} consider a multidimensional version of the Poisson optimal stopping problem (with constant stopping rate) and consider the sequence $\{V_\lambda^{(n)}\}_{n \geq 0}$. They observe that $V_\lambda^{(n)}$ is increasing in $n$ and show, under an assumption that a certain iterated expectation is finite, that $V_\lambda^{(n)}$ converges to $V^{(\infty)}_\lambda = V_\lambda$ geometrically fast. We work in one-dimension but allow for stopping opportunities arising from a state-dependent Poisson process.

Since $V_\theta^{(n)}$ is increasing in $n$ there must exist a limit which is finite on $\sI$ since $V^{(n)}_\theta < w$. Moreover, by monotone convergence
\begin{equation}
\label{eq:mono}
V_\theta^{(\infty)}( x) = \lim_n \E^x \left[ e^{-\beta T^\theta_1} \left\{ g(X_{ T^\theta_1}) \vee V^{(n)}_\theta (X_{T^\theta_1}) \right\} \right] = \E^x \left[ e^{-\beta T^\theta_1} \left\{ g(X_{ T^\theta_1}) \vee V^{(\infty)}_\theta (X_{T^\theta_1}) \right\} \right].
\end{equation}





In this section we are interested in the shape of the value function $V^{(\infty)}_\theta$. We saw some preliminary results in this direction in Section~\ref{ssec:firstresults}. In Theorem~\ref{thm:simpleincreasing} we saw that if both $g$ and $\theta$ are increasing then so is $V_\theta$; in Theorem~\ref{thm:hobson} we saw that if $g$ is convex and the arrival rate of the Poisson process is constant then $G_\lambda$ is convex.
In section we argue that it is not the shape of $g$ which is crucial, but rather the monotonicity/convexity properties of $\Psi$ where $\Psi : \sI \mapsto \R_+$ is given by
\[ \Psi(x)=\frac{g(x) \theta(x)}{\beta + \theta(x)}. \]
In particular, if $\Psi$ and $\theta$ are increasing then $G_\theta$ (Corollary~\ref{cor:Ginc}) and $V^{(\infty)}_\theta$ (Theorem~\ref{thm:finalinc}) are increasing, and if $\Psi$ is convex then $V^{(\infty)}_\theta$ is convex (Theorem~\ref{thm:Vconv}). (In the next section we give conditions under which $V^{(\infty)}_\theta = V_\theta$, and then monotonicity and convexity of $V_\theta$ are inherited from $V^{(\infty)}_\theta$.) Since $g$ increasing and $\theta$ increasing implies $\Psi$ is increasing, and $g$ convex and $\theta$ constant implies $\Psi$ is convex, the results of this and the next section include the results of Section~\ref{ssec:firstresults} as special cases, albeit under slightly stronger assumptions.

\begin{lem}\label{lem:nonexplode}
Suppose $(\theta/a^{2})$ is locally integrable, and further that
if an endpoint $e \in \{ \ell, r\}$ is attainable, then $\int_e \theta(x) \frac{|s(x)-s(e)|}{s'(x) a(x)^2} dx < \infty$ and $\theta(e) \in [0,\infty)$.

Then $\int_0^t \theta(X_u)du$ does not explode and $T^\theta_n \uparrow \infty$ almost surely.
\end{lem}

\begin{proof}
Fix $c$ in the interior of $\sI$ and define $s(x) = \int_c^x dy \exp \left( - \int_c^y \frac{b(z)}{a(z)^2} dz \right)$. Then $s$ is a scale function for $X$ and $M=s(X)$ is a local martingale with $dM_t = \eta(M_t) dB_t$ where $\eta(\cdot) = (s'a) \circ s^{-1} (\cdot)$. Let $W$ be a Brownian motion started at $s(x_0)$, let $H = H^W_{s(\ell),s(r)} = \inf \{u : W_u \notin (s(\ell),s(r)) \}$ and define $\Phi_u = \int_{0}^{u} \eta(W_s)^{-2} ds$ on $u < H$ with $\Phi_u = \infty$ on $u \geq H$. Then, by the Occupation Times Formula (Revuz and Yor~\cite[VI.1.6]{RevuzYor:99}), for $u < H$ we have
\[ \Phi_u = \int_{s(\sI)} \frac{1}{\eta(w)^2} L^{W,w}_{u} dw  = \int_{\sI} \frac{1}{s'(x) a(x)^2} L^{W,s(x)}_{u} dx \]
where $L^{W,w}_s$ is the local time of $W$ at $w$ by time $s$.
Necessarily $\Phi$ is strictly increasing and increases to infinity.

Let $A$ be inverse to $\Phi$ and let $M_t = W_{A_t}$. Then $A$ does not explode in finite time and $M$ solves $dM_t = \eta(M_t) dB_t$ for some Brownian motion $B$. Finally, let $X = s^{-1}(M)$. Then $X$ solves \eqref{eq:Xsde} with $X_0=s^{-1}(M_0) = s^{-1}(W_0) =x_0$.

Now, with this set-up, for $t \leq H^X_{\ell,r} = \inf \{t : X_t \notin (\ell,r) \}$ (note that $H^W_{s(\ell),s(r)} = A_{H^X_{\ell,r}}$),
\begin{equation}
\label{eq:nonexplode}
 \int_0^t \theta(X_u)du = \int_0^t \theta \circ s^{-1}(W_{A_u})du 
= \int_{s(\sI)} \frac{\theta \circ s^{-1}(w)}{\eta(w)^2} L^{W,w}_{A_t} dw = \int_{\ell}^r \frac{\theta(x)}{s'(x)a(x)^2} L^{W,s(x)}_{A_t} dx
\end{equation}
and for $t >  H^X_{\ell,r}$,
\begin{equation}
\label{eq:nonexplode2}
 \int_0^t \theta(X_u)du =  \int_{\ell}^r \frac{\theta(x)}{s'(x)a(x)^2} L^{W,s(x)}_{H^W_{s(\ell),s(r)}} dx + \theta(\ell)[t-H^X_\ell]^+  + \theta(r)[t-H^X_r]^+ .
\end{equation}
In particular, if both boundaries are natural, then using the fact that $s'$ is bounded on compact subsets of $(\ell,r)$ and $\theta/a^2$ is locally integrable we conclude from \eqref{eq:nonexplode} that $\int_0^t \theta(X_u)du$ is finite almost surely for each $t$. If one or more boundaries of $\sI$ is accessible (say $\ell$) then the same conclusion follows from the fact that for $x_0 > \ell$, $\E^{W_0=s(x_0)}[L^{W,s(x)}_{H^W_{s(\ell),s(r)}}] < \E^{W_0=s(x_0)}[L^{W,s(x)}_{H^W_{s(\ell),\infty}}] = 2 [(s(x) \wedge s(x_0))-s(\ell)]$ and hence $\E^{X_0=x_0} \left[ \int_{\ell} \frac{\theta(x)}{s'(x)a(x)^2} L^{W,s(x)}_{H^W_{s(\ell),s(r)}} dx \right] \leq 2 \int_\ell \frac{\theta(x)}{s'(x) a(x)^2} (s(x)-s(\ell)) dx < \infty$.

Let $\Gamma$ be random, and let $N$ be a Poisson process which is independent of $\Gamma$. It is easily seen that $N_\Gamma < \infty$ almost surely if and only if $\Gamma<\infty$ almost surely. It follows that $N_{\int_0^t \theta(X_u)du} < \infty$ for each $t$ almost surely and equivalently $(T^\theta_n)_{n \geq 1}$ increases to infinity almost surely.

\end{proof}
In addition to Standing Assumptions~\ref{sass:1} and \ref{sass:2}, for the rest of the paper we assume
\begin{sass}
\label{sass:3}
$(\theta/a^{2})$ is locally integrable. 
If an endpoint $e \in \{ \ell, r\}$ is attainable, then $\int_e \theta(x) \frac{|s(x)-s(e)|}{s'(x) a(x)^2} dx < \infty$ and $\theta(e) \in [0,\infty)$.
\end{sass}

For $h:\sI \mapsto \R_+$ define $\Psi_h : \sI \mapsto \R_+$ by $\Psi_h(x) = \frac{h(x)\theta(x)}{\beta + \theta(x)}$. Then $\Psi = \Psi_g$.
\begin{lem}
\label{lem:inch}
Let $Y$ solve
\[ dY_s   =  \frac{a(Y_s)}{ \sqrt{\beta + \theta(Y_s)} } dW_s + \frac{b(Y_s)}{\beta + \theta(Y_s) } ds \]
with initial condition $Y_0=x$.
Then
\begin{equation}
\label{eq:Psih}
 \E^x[ e^{-\beta T_1^\theta} h(X_{T^\theta_1})] = \E^x[\Psi_h(Y_T)]
\end{equation}
where $T$ is a unit-rate exponential random variable which is independent of $Y$.
\end{lem}

\begin{proof}

Let $C=(C_t)_{t \geq 0}$ be given by $C_t = \int_0^t (\beta + \theta(X^x_s)) ds$. 
Then by the local integrability assumption on $\theta/a^2$ of Standing Assumption~\ref{sass:3} we have that $C$ increases to infinity, but does not explode in finite time. 

Let $\Lambda$ be inverse to $C$. Our assumptions give us that $\Lambda_u < \infty$ for all finite $u$.
Let $Y$ be given by $Y_s = X_{\Lambda_s}$. Then $\frac{d \Lambda}{du} = \frac{1}{\beta + \theta(X_{\Lambda_u})} = \frac{1}{\beta + \theta(Y_u)}$. Moreover $Y$ is a time-homogeneous diffusion solving the SDE
\begin{equation}
\label{eq:timechange}
 dY_s = dX_{\Lambda_s}  =  a(X_{\Lambda_s}) dB_{\Lambda_s} + b(X_{\Lambda_s}) d \Lambda_s =  \frac{a(Y_s)}{ \sqrt{\beta + \theta(Y_s)} } d \tilde{B}_s + \frac{b(Y_s)}{\beta + \theta(Y_s) } ds
 \end{equation}
where $\tilde{B}$ is a Brownian motion given by $\tilde{B}_t = \int_0^t \left( \frac{ d \Lambda_s }{ds} \right)^{-1/2} dB_{\Lambda_s}$.
Note that since $(\beta + \theta)/a^2$ is locally integrable, $Y$ is unique in law.

Conversely, given $Y$ solving $dY_s = \frac{a(Y_s)}{ \sqrt{\beta + \theta(Y_s)} } d \hat{B}_s + \frac{b(Y_s)}{\beta + \theta(Y_s) } ds$ we can define $\Lambda_u = \int_0^u \frac{dv}{\beta + \theta(Y_u)}$, $C = \Lambda^{-1}$ and $X_s = Y_{C_s}$. Then $dX_s = a(X_s) dW_s + b(X_s) ds$.

Conditioning on the first event of the Poisson process we have that
\begin{eqnarray*}
\E^x \left[ e^{- \beta T^\theta_1} h(X_{T^\theta_1})  \right]
 & = & \E^x \left[ \int_0^\infty \theta(X_t) e^{- C_t} h(X_t) dt \right] \\
&  = &  \E^x \left[ \int_0^\infty \theta(X_{\Lambda_u}) e^{- u} h(X_{\Lambda_u}) d \Lambda_u \right] \\
& = &  \E^x \left[ \int_0^\infty  e^{- u} \frac{h(Y_u) \theta(Y_u)}{\beta + \theta(Y_u)} du \right]
 =  \E^x \left[ \int_0^\infty  e^{- u} \Psi_h(Y_u) du \right] = \E^x[ \Psi_h(Y_T)].
\end{eqnarray*}
\end{proof}

\begin{cor}
\label{cor:Ginc}
Suppose $\Psi$ is increasing. 
Then $G_\theta$ is increasing in $x$.
\end{cor}

\begin{proof}
Fix $x<y$. Let $Y^x$ and $Y^y$ denote solutions of \eqref{eq:timechange} started at $x$ and $y$ respectively. Since $Y^x$ and $Y^y$ are diffusions which are unique in law, there exists a coupling such that $Y^x \leq Y^y$ pathwise (recall Theorem~\ref{thm:simpleincreasing}).
In particular, there is a coupling such that $Y$ is increasing in its initial value on each sample path, and it follows that for $x<y$ and any increasing $\psi$, $\psi(Y^x_T) \leq \psi(Y^y_T)$. Then applying Lemma~\ref{lem:inch} with $h=g$ and $\Psi_g = \Psi$,
$G_\theta(x) =\E[ \Psi(Y^x_T) ] \leq \E[\Psi(Y^y_T)] = G_\theta(y)$.



\end{proof}


\begin{eg} Let $X$ be a diffusion in natural scale on $[0,\infty)$ or $(0,\infty)$. Let $g(x)= 1+x$ and suppose $\theta(x) = \beta/(1+2x)$, with $\theta(0)=\beta$ if $0$ is attainable. Then $\Psi=1/2$ and hence $G_\theta(x) = \frac{1}{2} \leq g$. Furthermore, applying Proposition~\ref{prop:V=G} we conclude that $\tau=T^\theta_1$ is optimal and  $V_\theta^{(\infty)}=G_\theta = \frac{1}{2}$.
\end{eg}

\begin{thm}
\label{thm:finalinc}
Suppose $\theta$ and $\Psi$ are increasing. Then $V_\theta^{(\infty)}$ is increasing.
\end{thm}

\begin{proof}
By Lemma~\ref{lem:inch} 
\begin{equation}
\label{eq:Psi^V}
V_\theta^{(n+1)}(x) = \E^x \left[ e^{-\beta T^\theta_1} ( g \vee V^{(n)}_\theta )(X_{T^\theta_1}) \right] =\E^x[ \Psi_{g \vee V_\theta^{(n)}}(Y_T)]
\end{equation}
where $\Psi_{g \vee v}(y) = \frac{(g(y) \vee v(y)) \theta(y)}{\beta + \theta(y)} = \Psi(y) \vee \frac{v(y) \theta(y)}{\beta + \theta(y)}$.

If $V_\theta^{(n)}$ is increasing, then since $\theta$ and $\Psi$ are also increasing, $\Psi_{g \vee V_\theta^{(n)}}$ is increasing. Then, using \eqref{eq:Psi^V} and a coupling argument as in the proof of Theorem~\ref{thm:simpleincreasing}, $V_\theta^{(n+1)}$ is increasing. Hence, since by Corollary~\ref{cor:Ginc} we have $V^{(1)}_\theta = G_\theta$ is increasing, we have by induction that $V_\theta^{(k)}$ is increasing for each $k$. The increasing limit of increasing functions is increasing. Hence $V_\theta^{(\infty)}$ is increasing.
\end{proof}

Now we turn to the issue of convexity. Since we do not expect convexity unless $X$ is in natural scale, for the rest of this section we suppose that $X$ is in natural scale. Note that in the results that follow there are assumptions on $\Psi$, but unlike Corollary~\ref{cor:Ginc} and Theorem~\ref{thm:finalinc}, there are no separate assumptions on $\theta$.

\begin{prop}
\label{prop:Gconv} Suppose that $X$ is in natural scale.
Suppose that if $\ell = -\infty$ then $\int_{-\infty} \frac{|y|(\beta + \theta(y))}{a(y)^2} dy = \infty$ and if $r = +\infty$ then $\int^{\infty} \frac{y(\beta + \theta(y))}{a(y)^2} dy = \infty$. There is no condition at finite endpoints.

Suppose $\Psi$ is convex. Then $G_\theta$ is convex in $x$ and $G_\theta \geq \Psi$.

Alternatively, if $\Psi$ is concave then $G_\theta$ is concave and $G_\theta \leq \Psi$.
\end{prop}

\begin{proof}
By a result of Kotani~\cite{Kotani:06} the conditions at the boundaries are exactly sufficient to guarantee that $Y$ given by $dY_s = dX_{\Lambda_s}=\frac{a(Y_s)}{\sqrt{\beta + \theta(Y_s)}} dB_s$ is a martingale. The result then follows from the representation in \eqref{eq:Psih} and Theorem~\ref{thm:hobson} (or Remark~\ref{rem:concave} in the case of concavity) with $\beta=0$.
\end{proof}

\begin{rem} The martingale property is essential here, and it is easy to construct a counterexample in the strict local martingale case using a linear payoff and a three-dimensional Bessel process.
\end{rem}

\begin{thm}
\label{thm:Vconv} Suppose that $X$ is in natural scale.
Suppose that if $\ell = -\infty$ then $\int_{-\infty} \frac{|y|(\beta + \theta(y))}{a(y)^2} dy = \infty$ and similarly if $r = +\infty$ then $\int^{\infty} \frac{y(\beta + \theta(y))}{a(y)^2} dy = \infty$. There is no condition at finite endpoints.


Suppose $\Psi$ is convex. Then $V^{(\infty)}_\theta$ is convex.

Suppose $\Psi$ is concave. Then $V^{(\infty)}_\theta$ is concave.
\end{thm}

\begin{proof}
Suppose the conditions of the theorem hold and $\Psi$ is convex. By Proposition~\ref{prop:Gconv}, $V^{(1)}_\theta \equiv G_\theta \geq \Psi$.

Suppose inductively that $V^{n}_\theta \geq V^{(n-1)}_\theta \geq \ldots \geq V^{(1)}_\theta =  G_\theta \geq \Psi$.

Consider $V^{(n+1)}_\theta(y)$. By  \eqref{eq:Psi^V} we have $V^{(n+1)}_\theta(y)= \E^y[\Psi_{g \vee V^{(n)}_\theta}(Y_T)]$.
Since $V_\theta^{(n)} \geq V_\theta^{(n-1)}$ it follows that $V^{(n+1)}_\theta \geq V^{(n)}_\theta \geq  \ldots \geq V^{(1)}_\theta = G_\theta \geq \Psi(y)$. Moreover, $V_\theta^{(n+1)} \geq V_\theta^{(n)} \geq \frac{ V_\theta^{(n)}(y)\theta(y)}{\beta + \theta(y)}$.
In particular, $V^{(n+1)}_\theta \geq \Psi \vee \frac{ V_\theta^{(n)} \theta}{\beta + \theta} = \Psi_{g \vee V^{(n)}_\theta}$. Thus,
$V^{(n+1)}_\theta(y) = \E^y[\Psi_{g \vee V^{(n)}_\theta} (Y_T)] \geq \Psi_{g \vee V^{(n)}_\theta}(y)$, and by Proposition~\ref{prop:cuteconvexity} with $c = \Psi_{g \vee V^{(n)}_\theta}$ we conclude that $V^{(n+1)}_\theta$ is convex.



Finally, since the increasing limit of convex functions is convex we conclude that $V^{(\infty)}_\theta$ is convex.

The corresponding result for concavity is more direct: if $\Psi$ is concave then $G_\theta \leq \Psi = \frac{g \theta}{\beta+\theta} \leq g$. Then $V^{(n)}_\theta= V^{(1)}_\theta= G_\theta$ and $V^{(\infty)}_\theta = G_\theta$. Since $G_\theta$ is concave by Proposition~\ref{prop:Gconv}, the result follows.

\end{proof}

\section{Monotonicity and convexity of $V_\theta$}
\label{sec:Vinfty}
Standing Assumptions~\ref{sass:1}, \ref{sass:2} and \ref{sass:3} remain in force.

\begin{prop}
\label{prop:unique}
Suppose
\begin{equation}
\label{eq:gcondition}
\E \left[ \sup_{s \geq 0} \left\{ e^{-\beta s}g(X_s) \right\} \right] < \infty
\hspace{10mm} \mbox{and} \hspace{10mm}
\lim_{t \uparrow \infty} \E \left[ \sup_{s \geq t} \left\{ e^{-\beta s}g(X_s) \right\} \right] = 0.
\end{equation}

Then $V_\theta(x) = V^{(\infty)}_\theta(x)$.
\end{prop}

\begin{rem}
From the discussion at the end of Example~\ref{eg:DW} we know that \eqref{eq:gcondition} holds in that setting, and clearly it also holds whenever $g$ is bounded. Indeed it holds for all the examples in Section~\ref{sec:eg}.
\end{rem}

\begin{proof}
Let $K = \E \left[ \sup_{s \geq 0} \left\{ e^{-\beta s}g(X_s) \right\} \right]$. Given $\epsilon>0$, choose $t_0$ such that
$\E \left[ \sup_{s \geq t_0} \left\{ e^{-\beta s}g(X_s) \right\} \right] < \epsilon/2$, and, recalling that by Lemma~\ref{lem:nonexplode} we have that $T_n^\theta \uparrow \infty$ almost surely, choose $n_0$ such that $\Prob(T^\theta_{n_0} \leq t_0) < \frac{\epsilon}{2K}$.

Then, for any stopping time $\tau \in \sT(\T_\theta)$ and $n \geq n_0$,
\begin{eqnarray*}
\E^x \left[ I_{ \{ \tau > T^\theta_{n} \} } e^{-\beta \tau} g(X_\tau) \right]
& = & \E^x \left[ I_{ \{ \tau > T^\theta_{n} > t_0 \} } e^{-\beta \tau} g(X_\tau) \right] + \E^x \left[ I_{ \{ \tau > T^\theta_{n} \} } I_{ \{ t_0 \geq T^\theta_{n} \} } e^{-\beta \tau} g(X_\tau) \right] \\
& \leq & \E^x \left[ I_{ \{ \tau > t_0 \} } e^{-\beta \tau} g(X_\tau) \right] + \E^x\left[ I_{ \{ t_0 \geq T^\theta_{n} \} } \sup_{s \geq 0} \left\{ e^{-\beta s} g(X_s) \right\} \right] \\
& \leq & \E^x \left[ \sup_{s \geq t_0}  \left\{ e^{-\beta s} g(X_s) \right\} \right] + \Prob^x( T^\theta_{n} \leq t_0 ) \E^x \left[ \sup_{s \geq 0} \left\{ e^{-\beta s} g(X_s) \right\} \right] \\
& < & \frac{\epsilon}{2} + \frac{\epsilon}{2K}{K} = \epsilon.
\end{eqnarray*}
It follows that
\begin{eqnarray*}
\sup_{\tau \in \sT(\T_\theta)} \E^x\left[ e^{-\beta \tau} g(X_\tau) \right]
& = & \sup_{\tau \in \sT(\T_\theta)}\left\{  \E^x\left[ e^{-\beta \tau} g(X_\tau) I_{ \{ \tau \leq T_{n} \} } \right] +
  \E^x\left[ e^{-\beta \tau} g(X_\tau) I_{ \{ \tau > T_{n} \} } \right] \right\} \\
& \leq &  \sup_{\tau \in \sT( \{ T^\theta_1, \ldots T^\theta_{n})} \E^x\left[ e^{-\beta \tau} g(X_\tau) \right] + \epsilon \\
& = & V^{(n)}_\theta(x) + \epsilon
\end{eqnarray*}
Hence, for large enough $n$, $V^{(n)}_\theta(x) \leq V_\theta(x) \leq V^{(n)}_\theta(x) + \epsilon$. Taking limits we find $V_\theta = V^{(\infty)}_\theta$.
\end{proof}

Combining Proposition~\ref{prop:unique} and Theorem~\ref{thm:finalinc} we obtain:
\begin{cor}
\label{cor:Vinc}
Suppose $\theta$ and $\Psi$ are increasing and that \eqref{eq:gcondition} holds. Then $V_\theta$ is increasing.
\end{cor}

Combining Theorem~\ref{thm:Vconv} and Proposition~\ref{prop:unique} we obtain the corresponding result for $V_\theta$:

\begin{cor}
\label{cor:Vconv} Suppose that $X$ is in natural scale.
Suppose that if $\ell = -\infty$ then $\int_{-\infty} \frac{|y|(\beta + \theta(y))}{a(y)^2} dy = \infty$ and similarly if $r = +\infty$ then $\int^{\infty} \frac{y(\beta + \theta(y))}{a(y)^2} dy = \infty$. There is no condition at finite endpoints. Suppose that \eqref{eq:gcondition} holds.

Suppose $\Psi$ is convex. Then $V_\theta$ is convex.

Suppose $\Psi$ is concave. Then $V_\theta$ is concave.
\end{cor}


\begin{thebibliography}{90}

\bibitem{Alvarez:03} Alvarez, L. (2003): On the convexity and comparative static properties of $r$-harmonic and $r$-excessive mapping for a class of diffusions, {\em Annals of Applied Probability}m {\bf 13}(4) 1517-1533.

\bibitem{BergmanGrundyWiener:96} Bergman, Y., B. Grundy and Z. Wiener, (1996): On the theory of option pricing, {\em Journal of Finance}, {\bf 51} 1573-1610.

\bibitem{BorodinSalminen:02} Borodin, A. and P. Salminen (2002): {\em Handbook of {B}rownian motion - facts and formul{\ae}.} Birkh\"auser, Basel.


\bibitem{CoxRoss:76} Cox J.C and S. Ross (1976): The valuation of options for alternative stochastic processes. {\em J. Financial Economics} {\bf 3}, 145-166.


\bibitem{DupuisWang:02} Dupuis, P. and H. Wang (2002): Optimal stopping with random intervention times, {\em Advances in Applied Probability} 34(1), 141-157.

\bibitem{ElKarouiJeanblancShreve:98} El Karoui, N., M. Jeanblanc-Picqu{\'e} and S.E. Shreve (1998): Robustness of the {B}lack-{S}choles formula, {\em Mathematical Finance}, {\bf 8}(2), 93-126.

\bibitem{Ekstrom:04} Ekstr{\o}m, E. (2004): Properties of American option prices. {\em Stochastic Processes and Applications}, {\bf 114}, 265-278.

\bibitem{EnglebertSchmidt:84} Engelbert, H.J. and W. Schmidt (1984): On one-dimensional stochastic differential equations with generalised drift. {\em Lecture Notes in Control and Information Sciences} {\bf 69} 143-155, Springer-Verlag, Berlin.

\bibitem{HendersonSunWhalley:14} Henderson, V., J. Sun and A.E. Whalley (2014): Portfolios of American options under general preferences: results and counterexamples, {\em Mathematical Finance}, {\bf 24}(3), 533-566.


\bibitem{Hobson:98} Hobson, D.G. (1998): Volatiltiy misspecification, option pricing and superreplication via coupling, {\em Annals of Applied Probability}, {\bf 8}(1) 193-205.

\bibitem{HobsonZeng:19} Hobson, D.G. and M. Zeng (2019): Constrained optimal stopping, liquidity and effort. {\em Stochastic Processes and Applications}, to appear. DOI:10.1016/j.spa.2019.10.010 arXiv:1901.07270

\bibitem{KaratzasShreve:91}  Karatzas, I. and  S.E. Shreve: (1991): {\em Brownian Motion and Stochastic Calculus}. Second Edition, Springer, New York.


\bibitem{Kotani:06} Kotani, S. (2006): On a condition that one-dimensional diffusion processes are martingales. {\em Seminaire de Probabilit\'{e}s} XXXIX 149-156. Lecture Notes in Mathematics, 1874. Springer, Berlin.

\bibitem{LangeRalphStore:19} Lange, R.-J., D. Ralph and K. St{\o}re: (2020): Real-option valuation in multiple dimensions using Poisson optional stopping times. {\em J. Finan. Quant. Anal.}, {\bf 55}(2), 653-677.

\bibitem{Lempa:12} Lempa, J. (2012): Optimal stopping with information constraint. {\em Appl. Math. Optim.}, {\bf 66}(2) p147--173.


\bibitem{LiangWei:16} Liang, G. and  W. Wei (2016): Optimal switching at Poisson random intervention times, {\em Discrete and Continuous Dynamical Systems, Series B} 21(5), 1483-1505.

\bibitem{Lindvall:92} Lindvall, T. (1992): {\em Lectures on the coupling method}. Wiley, New York.

\bibitem{MenaldiRobin:16} Menaldi, J.L. and M. Robin: (2016): On some Optimal stopping problems with constraint, {\em SIAM Journal on Control and Optimization} 54(5), 2650-2671.


\bibitem{Merton:73} Merton, R.C. (1973): The theory of rational option pricing. {\em Bell J. Econ. Manage. Sci}, {\bf 4}, 141-183.


\bibitem{RevuzYor:99} Revuz, D. and M. Yor (1999): {\em Continuous martingales and Brownian motion}. Third Edition, Springer, Berlin.

\bibitem{RogersWilliams:01} Rogers, L.C.G and D. Williams (2001): {\em Diffusions, {M}arkov processes and martingales, Vol. 1}. Cambridge University Press, Cambridge.

\bibitem{RogersZane:98} Rogers, L.C.G. and O. Zane (1998): A simple model of liquidity effects, {\em Advances in Finance and Stochastics. Essays in Honour of Dieter Sondermann. Eds K Sandmann and P, Sch\"{o}nbucher.} Springer, Berlin.



\end{thebibliography}
\end{document}